\documentclass[12pt,reqno]{amsart}
\usepackage{graphicx}
\usepackage{verbatim}
\usepackage{textcomp}
\usepackage{amssymb}
\usepackage{cite}
\usepackage{amsmath}
\usepackage{latexsym}
\usepackage{amscd}
\usepackage{amsthm}
\usepackage{mathrsfs}
\usepackage{bm}
\usepackage{url}
\usepackage{hyperref}
\usepackage{bookmark}
\allowdisplaybreaks[3]
\newtheorem{thm}{Theorem}[section]
\newtheorem{corr}[thm]{Corollary}
\newtheorem{lem}[thm]{Lemma}
\newtheorem{prop}[thm]{Proposition}

\theoremstyle{definition}

\theoremstyle{remark}
\newtheorem{rem}{Remark}[section]
\numberwithin{equation}{section}
\begin{document}
\title[Universal gradient estimates for solutions]
{Universal gradient estimates for solutions of $\Delta_{p,f}u+au^{\sigma}\ln u=0$ on complete Riemannian manifolds}

\author{Jingxu Liu$^{1}$}
\author{Zhen Wang$^{2}$}

\address{School of Mathematics Science, South China Normal University, Guangzhou, 510631, P. R. China.}
\address{School of Mathematics and Statistics, Henan Normal University, Xinxiang, 453007, P. R. China.}

\email{jxl2024@126.com (J. Liu)}
\email{wangz2025@126.com (Z. Wang)}

\thanks{Corresponding author: wangz2025@126.com (Z. Wang)}

\begin{abstract}
In this paper, we consider the weighted $p$-Laplacian equation
$$ \Delta_{p,f}u+au^{\sigma}\ln u=0$$
defined on a complete smooth metric measure space under the conditon that the $m$-Bakry-\'{E}mery Ricci curvature has a lower bound, where $a$, $\sigma$ are two nonzero real constants. By applying the Nash-Moser iteration, we obtain sharp gradient estimates and thereby establish Liouville theorems for the above equation.
\end{abstract}



\maketitle

\section{Introduction}
Gradient estimates play a crucial role not only in the study of elliptic equations on Riemannian manifolds (see \cite{Bidaut2020,LiYau1986,Yau75}), but also in establishing Liouville type theorems. In particular, through this methods, some mathematicians have successfully investigated gradient estimates for partial differential equations on smooth metric measure spaces and then obtained Liouville theorems, see\cite{Brighton2013,Huangli2014,HuangZhao2023,lix05,lih15,lu2023,Ma06}.

The following Lane-Emden equation
\begin{equation}\label{1-Int-112}
\Delta u+u^\alpha=0
\end{equation}
with $\alpha\in\mathbb{R}$, serves as a fundamental model of semilinear elliptic equations. In\cite{Gidas1981}, Gidas and Spruck established several fundamental results for \eqref{1-Int-112}, including Liouville theorem as follows: Let $(M^n, g)$ be an $n$-dimensional $(n>2)$ complete Riemannian manifold with nonnegative Ricci curvature, and let $u$ be a nonnegative $C^2$ solution of \eqref{1-Int-112} on $M^n$ with $\alpha \in (1, \frac{n+2}{n-2})$, then they obtained $u\equiv0$.

As far as we know, no proof of this result other than the original one by Gidas and Spruck has appeared in the literature. In addition, as part of their work on the classification of singularities, they derived a key decay estimate and Harnack inequality, which played an essential role in their singularity analysis.

Recently, one paid attention to studying the following equation
\begin{equation}\label{1-Int-111}
\Delta u+a u\ln u+bu=0
\end{equation}
which is closely related to the famous Log-Sobolev inequality. By replacing $u$ by $e^{-\frac{b}{a}}u$ with $a,b$ two real constants, the equation \eqref{1-Int-111} reduces to
\begin{equation}\label{1-Int-222}
\Delta u+a u\ln u=0.
\end{equation}
Therefore, the study to \eqref{1-Int-222} is very interesting. The authors\cite{WJ24} obtained gradient estimates and a local Liouville theorem for nonlinear elliptic equations \eqref{1-Int-222}.

The $m$-Bakry-\'{E}mery Ricci curvature is defined as
\begin{equation}\label{1-Int-1}
\mathrm{Ric}_{f,m}=\mathrm{Ric}+\nabla^{2}f-\frac{1}{m-n}df\otimes df,
\end{equation}
where $m\geq n$ and $m=n$ if and only if $f$ is a constant. In this paper, we consider the weighted $p$-Laplacian equation
\begin{equation}\label{1-Int-2}
\Delta_{p,f}u+au^{\sigma}\ln u=0
\end{equation}
on an $n$-dimensional complete Riemannian manifold $(M^{n},g,d\mu)$, where $a$, $\sigma$ are two nonzero real constants and
\begin{equation}\label{1-Int-3}
\Delta_{p,f}u:=e^{f}\mathrm{div}(e^{-f}\lvert\nabla u\rvert^{p-2}\nabla u)
\end{equation}
is the weighted $p$-Laplacian with $p>1$. Obviously, \eqref{1-Int-2} can be seen as a generalization to \eqref{1-Int-222} (when $m=n$, $\sigma=1$ and $p=2$, the equation \eqref{1-Int-2} becomes
\eqref{1-Int-222}).

For $p$-Laplacian equations, under the sectional curvature conditions, some mathematicians (see\cite{Chen09,HuangLi2013,Wang23,
Yang08} and the references therein) obtain gradient estimates and Liouville theorems by using
the Hessian comparison theorem since the $p$-Laplacian is nonlinear. Through the above investigation, we see that for a class of $p$-Laplacian equations on complete Riemannian manifolds, the method of constructing cutoff functions only yields related gradient estimate results under sectional curvature conditions. This limitation arises because the construction of the distance function relies on the Hessian comparison theorem, which inevitably requires sectional curvature bounds. Our goal is to bypass the Hessian comparison theorem using the Nash-Moser iteration method, thereby obtaining results under the Ricci curvature condition
(for example \cite{WJ24,Wz2011,WYDWGD2026}).

Inspired by \cite{HeWangWei2024,WYWG2023,Song2025}, in what follows, we extend our study based on the equation \eqref{1-Int-2} and obtain the following results by using of Nash-Moser iteration:

\begin{thm}\label{thm-1}
Let $(M^{n},g,d\mu)$ be an $n$-dimensional $(n>2)$ complete smooth metric measure space with $\mathrm{Ric}_{f,m}\geq-(m-1)Kg$, where $K$ is a non-negative constant. Assume that $u$ is a positive weak solution to equation \eqref{1-Int-2} on the geodesic ball $B_{x_{0}}(R)\subset M$ with $a<0$, $p>1$ and $$0<\sigma<\frac{(m+2)(p-1)}{m},$$
then it holds
\begin{align}\label{1-Th-1}
\frac{\lvert\nabla u\rvert}{u^{1-\frac{m[p-(\sigma+1)]+2(p-1)}{(m+2)(p-1)}}}\leq\frac{C(1+\sqrt{K}R)}{R} \quad on \ B_{x_{0}}(\frac{R}{2}),
\end{align}
where $C=C(p,m,\sup_{B_{x_{0}}(R)} u)$ depends only on $p$, $m$ and $\sup_{B_{x_{0}}(R)} u$.

If we further assume that $p>m+1$, then \eqref{1-Th-1} still holds for $a<0$ and
$$\sigma\in \Big(-\infty,-\frac{(m+2)(p-1)}{p-(m+1)}\Big).$$
\end{thm}

In particular by letting $R\rightarrow \infty$, we obtain the following Liouville theorem immediately:

\begin{corr}\label{corr-2}
Let $(M^{n},g,d\mu)$ be an $n$-dimensional $(n>2)$ complete smooth metric measure space with $\mathrm{Ric}_{f,m}\geq0$. For $a<0$, $p>1$ and $$0<\sigma<\frac{(m+2)(p-1)}{m},$$
then there does not exist any positive nontrivial bounded solution to equation \eqref{1-Int-2}.

If we further assume that $p>m+1$, for $a<0$ and
$$\sigma\in \Big(-\infty,-\frac{(m+2)(p-1)}{p-(m+1)}\Big),$$
then nor does there exist any positive nontrivial bounded solution to the equation \eqref{1-Int-2}.
\end{corr}

\begin{rem}\label{rem-1}
In particular, if $\sigma=p-1\in\Big(0,\frac{(m+2)(p-1)}{m}\Big)$, then the equation \eqref{1-Int-2} becomes
$$
\Delta_{p,f}u+au^{p-1}\ln u=0,
$$
then our results is similar to that in \cite{Huang2026} (see Theorem 1.1) but improve the range $p>2$ to $p>1$.
\end{rem}
\begin{rem}\label{rem-2}
Compared with the previous work \cite{HLZ2026}, when $\sigma=1$ our equation \eqref{1-Int-2} becomes
$$
\Delta_{p,f}u+au\ln u=0,
$$
which is same as the equation in \cite{HLZ2026} when $a(x)\equiv C$ (see equation(1.3)), but our results improve the range $p>2$ to $p>1$.
\end{rem}
\section{Proof of results}
It is well-known that a function $u$ is said to be a positive weak solution to equation \eqref{1-Int-2} if $u$ satisfies
$$\int_M|\nabla u|^{p-2}\langle\nabla u,\nabla \phi\rangle d\mu=\int_Ma(x)(u+B)^\alpha[\ln(u+C)]^\sigma\phi d\mu$$
for all $\phi\in C_{0}^\infty(M)$. Denote $h=u^{\epsilon}$, where $\epsilon\neq0$ is a constant to be determined later. The weighted elliptic operator $\mathcal{L}_{f}$ is defined by
$$\mathcal{L}_{f}=e^{f}\mathrm{div}(e^{-f}|\nabla h|^{p-2}A(\nabla\cdot)),$$
where
$$A=Id+(p-2)\frac{\nabla h\otimes\nabla h}{|\nabla h|^2}.$$
Because equation \eqref{1-Int-2} can be either degenerate or singular at the points $\{|\nabla u|=0\}$, which is $\{|\nabla h|=0\}$, we replace the linearized $\mathcal{L}_{f}$ with its approximate operator, i.e.,
$$\mathcal{L}_{f,\varepsilon}=e^{f}\mathrm{div}(e^{-f}\omega_{\varepsilon}^{\frac{p}{2}-1}A_{\varepsilon}
(\nabla\cdot)),$$
where $\varepsilon>0$, $\omega_{\varepsilon}=|\nabla h_{\varepsilon}|^2+\varepsilon$, $A_{\varepsilon}=Id+(p-2)\frac{\nabla h_{\varepsilon}\otimes\nabla h_{\varepsilon}}{|\nabla h|^2}$. From the gradient estimate in \cite{BK09}, we know that $u\in C^{1,\alpha}$ for some $\alpha>0$ and $u\in W^{1,\beta}$ for some $\beta>1$. In fact, $u$ is smooth away from $\{|\nabla u|=0\}$. In order to avoid tedious presentation, we omit the details. The interested reader can refer to \cite{BK09} for more details.

Then, we prove the following lemma:
\begin{lem}\label{lem-1}
Let $(M^n,g,d\mu)$ be a complete smooth metric measure space with $\mathrm{Ric}_{f,m}\geq-(m-1)Kg$, where $K$ is a nonnegative constant.
Let $u$ be a positive weak solution of \eqref{1-Int-2} and the constants $p$, $\sigma$ satisfy conditions of Theorem \ref{thm-1}, for $h=u^{\frac{m[p-(\sigma+1)]+2(p-1)}{(m+2)(p-1)}}$ and $Q=\lvert\nabla h\rvert^{p}$, then
\begin{align}
\frac{1}{p}\mathcal{L}_{f}(Q)\geq&\frac{\sigma^{2}m(p-1)^{2}+\sigma m^{2}(p-1)(p-\sigma-1)+2\sigma m(p-1)^{2}}{[m(p-\sigma-1)+2(p-1)]^{2}}\frac{Q^{2}}{h^{2}}\nonumber\\&-\frac{\sigma m(p-1)\nabla h\nabla Q}{m(p-\sigma-1)+2(p-1)}\frac{Q^{1-\frac{2}{p}}}{h}-(m-1)KQ^{2-\frac{2}{p}}\notag\\
&-a\big[\frac{m(p-\sigma-1))+2(p-1)}{(m+2)(p-1)}\big]^{p-2}h^{\frac{(2\sigma-m-2)(p-1)+m\sigma}{(m+2)(p-1)-m\sigma}}Q.
\end{align}
\end{lem}

\proof
We define $h=u^{\epsilon}$, where $\epsilon\neq0$ is a constant to be determined. By a straightforward computation, we have
\begin{align}\label{Prel-2}
\Delta_{p,f}h=&\Delta_{p,f}(u^{\epsilon})=e^{f}\mathrm{div}(e^{-f}\lvert \nabla(u^{\epsilon}) \rvert^{p-2}\nabla(u^{\epsilon}))\nonumber\\
=&\epsilon^{p-1}[(p-1)(\epsilon-1)u^{(p-1)(\epsilon-1)-1}\lvert \nabla u \rvert^{p}+u^{(p-1)(\epsilon-1)}\Delta_{p,f}u]\nonumber\\
=&\gamma\frac{\lvert \nabla h\rvert^{p}}{h}-a\epsilon^{p-2}h^{\gamma+\frac{\sigma}{\epsilon}}\ln h,
\end{align}
where $\gamma=\frac{(p-1)(\epsilon-1)}{\epsilon}$, which gives
\begin{align}\label{Prel-3}
\nabla h\nabla \Delta_{p,f}h=&\nabla h\nabla\Big(\gamma\frac{\lvert \nabla h\rvert^{p}}{h}-a\epsilon^{p-2}h^{\gamma+\frac{\sigma}{\epsilon}}\ln h\Big)\nonumber\\
=&\frac{\gamma}{h}\nabla h\nabla\lvert \nabla h \rvert^{p}-a\epsilon^{p-2}\Big(\gamma+\frac{\sigma}{\epsilon}\Big)h^{\gamma+\frac{\sigma}{\epsilon}-1}\ln h\lvert\nabla h\rvert^{2}\nonumber\\&-a\epsilon^{p-2}h^{\gamma+\frac{\sigma}{\epsilon}-1}\lvert\nabla h\rvert^{2}-\frac{\gamma}{h^{2}}\lvert \nabla h \rvert^{p+2}.
\end{align}
The weighted $p$-Bochner formula (see \cite{WL2016}) with respect to the function $h$ is
\begin{align}\label{Prel-4}
\frac{1}{p}\mathcal{L}_{f}(\lvert \nabla h \rvert^{p})=&\lvert \nabla h \rvert^{2p-4}(\lvert \mathrm{Hess h} \rvert_{A}^{2}+\mathrm{Ric}_{f}(\nabla h,\nabla h))+\lvert \nabla h \rvert^{p-2}\nabla h\nabla\Delta_{p,f}h\nonumber\\
\geq&\frac{(\Delta_{p,f}h)^{2}}{m}+\lvert \nabla h \rvert^{2p-4}\mathrm{Ric}_{f,m}(\nabla h,\nabla h)+\lvert \nabla h \rvert^{p-2}\nabla h\nabla\Delta_{p,f}h.
\end{align}
Inserting \eqref{Prel-3} into \eqref{Prel-4}, we have
\begin{align}\label{Prel-5}
\frac{1}{p}\mathcal{L}_{f}(\lvert \nabla h \rvert^{p})\geq&\frac{1}{m}\Big(\gamma\frac{\lvert \nabla h\rvert^{p}}{h}-a\epsilon^{p-2}h^{\gamma+\frac{\sigma}{\epsilon}}\ln h\Big)^{2}-(m-1)K\lvert \nabla h\rvert^{2p-2}\nonumber\\
&+\lvert \nabla h \rvert^{p-2}\Big[\frac{\gamma}{h}\nabla h\nabla\lvert \nabla h \rvert^{p}-a\epsilon^{p-2}\Big(\gamma+\frac{\sigma}{\epsilon}\Big)h^{\gamma+\frac{\sigma}{\epsilon}-1}\ln h\lvert\nabla h\rvert^{2}\nonumber\\
&-\frac{\gamma}{h^{2}}\lvert \nabla h \rvert^{p+2}-a\epsilon^{p-2}h^{\gamma+\frac{\sigma}{\epsilon}-1}\lvert\nabla h\rvert^{2}\Big]\nonumber\\
=&\Big(\frac{\gamma^{2}}{m}-\gamma\Big)\frac{\lvert \nabla h \rvert^{2p}}{h^{2}}+\frac{\gamma}{h}\lvert \nabla h \rvert^{p-2}(\nabla h\nabla\lvert \nabla h\rvert^{p})-(m-1)K\lvert \nabla h\rvert^{2p-2}\nonumber\\
&-a\epsilon^{p-2}\Big[\frac{\gamma(m+2)}{m}+\frac{\sigma}{\epsilon}\Big]h^{\gamma+\frac{\sigma}{\epsilon}-1}\ln h\lvert\nabla h\rvert^{p}-a\epsilon^{p-2}h^{\gamma+\frac{\sigma}{\epsilon}-1}\lvert\nabla h\rvert^{p}\nonumber\\
&+\frac{a^{2}}{m}\epsilon^{2(p-2)}h^{2(\gamma+\frac{\sigma}{\epsilon})}(\ln h)^{2}.
\end{align}
Now we choose $\epsilon=\frac{m[p-(\sigma+1)]+2(p-1)}{(m+2)(p-1)}$ such that $\frac{\gamma(m+2)}{m}+\frac{\sigma}{\epsilon}=0$, then $\gamma=\frac{-\sigma m(p-1)}{m[p-(\sigma+1)]+2(p-1)}$. To ensure $\epsilon>0$ and $\frac{\gamma^{2}}{m}-\gamma>0$, we need to let
$$p>1,\quad \sigma\in\Big(0,\frac{(m+2)(p-1)}{m}\Big)$$
or
$$p> m+1,\quad \sigma\in \Big(-\infty,-\frac{(m+2)(p-1)}{p-(m+1)}\Big) \cup \Big(0,\frac{(m+2)(p-1)}{m}\Big).$$
Then in this case, \eqref{Prel-5} becomes
\begin{align}\label{Prel-6}
\frac{1}{p}\mathcal{L}_{f}(\lvert \nabla h \rvert^{p})\geq&\Big(\frac{\gamma^{2}}{m}-\gamma\Big)\frac{\lvert \nabla h \rvert^{2p}}{h^{2}}+\frac{\gamma}{h}\lvert \nabla h \rvert^{p-2}(\nabla h\nabla\lvert \nabla h\rvert^{p})-(m-1)K\lvert \nabla h\rvert^{2p-2}\nonumber\\
&+\frac{a^{2}}{m}\epsilon^{2(p-2)}h^{2(\gamma+\frac{\sigma}{\epsilon})}(\ln h)^{2}-a\epsilon^{p-2}h^{\gamma+\frac{\sigma}{\epsilon}-1}\lvert\nabla h\rvert^{p}\nonumber\\
\geq&\frac{\gamma(\gamma-m)}{m}\frac{ Q^{2}}{h^{2}}+\frac{\gamma}{h}Q^{1-\frac{2}{p}}\nabla h\nabla Q-(m-1)KQ^{2-\frac{2}{p}}-a\epsilon^{p-2}h^{\gamma+\frac{\sigma}{\epsilon}-1}Q\nonumber\\
=&\frac{\sigma^{2}m(p-1)^{2}+\sigma m^{2}(p-1)(p-\sigma-1)+2\sigma m(p-1)^{2}}{[m(p-\sigma-1)+2(p-1)]^{2}}\frac{Q^{2}}{h^{2}}\nonumber\\&-\frac{\sigma m(p-1)\nabla h\nabla Q}{m(p-\sigma-1)+2(p-1)}\frac{Q^{1-\frac{2}{p}}}{h}-(m-1)KQ^{2-\frac{2}{p}}\nonumber\\&-a\big[\frac{m(p-\sigma-1))
+2(p-1)}{(m+2)(p-1)}\big]^{p-2}h^{\frac{(2\sigma-m-2)(p-1)+m\sigma}{(m+2)(p-1)-m\sigma}}Q.
\end{align}
So the proof of Lemma \ref{lem-1} is completed.

Now, let us recall the following Saloff-Coste's Sobolev embedding theorem which will be helpful very much in the proof of our results:
\begin{prop}
(\cite{Saloff1992,Zhao18}) Let $(M^n,g,d\mu)$ be an $n$-dimensional $(n>2)$ complete smooth metric measure space with $\mathrm{Ric}_{f,m}\geq -(m-1)Kg$, where $K$ is a nonnegative constant. For $m\geq3$, there exists some positive constant $C(m)$ depending only on $m$, such that for all $\Omega=B_{x_{0}}(R)\subset M$ and $\varphi \in C_{0}^{\infty}(\Omega)$, it holds that
\begin{equation}\label{Prel-7}
\Big(\int_{\Omega}\lvert \varphi\rvert^{\frac{2m}{m-2}}d\mu\Big)^{\frac{m-2}{m}}\leq e^{C(m)(1+\sqrt{K}R)}V_{f}^{-\frac{2}{m}}R^{2}\int_{\Omega}(\lvert \nabla\varphi\rvert^{2}+R^{-2}\varphi^{2})d\mu,
\end{equation}
where $V_{f}=\int_{\Omega}d\mu$.
\end{prop}

Now, we can prove the following integral inequality for the solutions to equation \eqref{1-Int-2}:

\begin{lem}\label{lem-2}
Let $\Omega=B_{x_{0}}(R)\subset M$ be a geodesic ball. Under the same assumptions as in Lemma \ref{lem-1}, by letting $M_{0}:=\sup_{B_{x_{0}}(R)} u$, for $a<0$ we have
\begin{align}\label{Prel-8}
&\frac{a_{5}}{2}\int_{\Omega}Q^{t+2}\eta^{2}d\mu+\frac{c_{2}}{t}e^{-C(m)(1+\sqrt{K}R)}V_{f}^{\frac{2}{m}}R^{-2}\Big(\int_{\Omega} Q^{(\frac{t}{2}+1-\frac{1}{p})\frac{2m}{m-2}}\eta^{\frac{2m}{m-2}}d\mu\Big)^{\frac{m-2}{m}}\nonumber\\\leq&\frac{c_{3}}{t}
\int_{\Omega}Q^{t+2-\frac{2}{p}}\lvert\nabla\eta\rvert^{2}d\mu
+c_{4}t_{0}^{2}R^{-2}\int_{\Omega}Q^{t+2-\frac{2}{p}}\eta^{2}d\mu,
\end{align}
where $c_{2},\ c_{3},\ c_{4}$ and $a_{5}$ depend only on $m, p, M_{0}$.
\end{lem}

\proof
For $a<0$, it follows from \eqref{lem-1} that
\begin{align}\label{Prel-9}
\frac{1}{p}\mathcal{L}_{f}(Q)\geq&\frac{\sigma^{2}m(p-1)^{2}+\sigma m^{2}(p-1)(p-\sigma-1)+2\sigma m(p-1)^{2}}{[m(p-\sigma-1)+2(p-1)]^{2}}\frac{Q^{2}}{h^{2}}\nonumber\\&-\frac{\sigma m(p-1)\nabla h\nabla Q}{m(p-\sigma-1)+2(p-1)}\frac{Q^{1-\frac{2}{p}}}{h}-(m-1)KQ^{2-\frac{2}{p}}.
\end{align}
Let $\psi=Q^{t}\eta^{2}$, where $\eta\in C_{0}^{\infty}(\Omega)$ is nonnegative, and $t>1$ is to be determined later. Then multiplying both sides of \eqref{Prel-9} with $\psi$ and integrating it, we obtain
\begin{align}\label{Prel-10}
&\int_{\Omega}[Q^{1-\frac{2}{p}}\nabla Q+(p-2)Q^{1-\frac{4}{p}}\langle\nabla Q,\nabla h\rangle\nabla h]\nabla(Q^{t}\eta^{2})d\mu+a_{1}\int_{\Omega}Q^{t+2}h^{-2}\eta^{2}d\mu\nonumber\\\leq
&\frac{\sigma pm(p-1)}{m(p-\sigma-1)+2(p-1)}\int_{\Omega}Q^{t+1-\frac{2}{p}}h^{-1}\eta^{2}\langle\nabla Q\nabla h\rangle d\mu\notag\\
&+p(m-1)K\int_{\Omega}Q^{t+2-\frac{2}{p}}\eta^{2} d\mu,
\end{align}
where $a_{1}=\frac{[\sigma^{2}m(p-1)^{2}+\sigma m^{2}(p-1)(p-\sigma-1)+2\sigma m(p-1)^{2}]p}{[m(p-\sigma-1)+2(p-1)]^{2}}$. Hence, we obtain
\begin{align}\label{Prel-11}
&t\int_{\Omega}Q^{t-\frac{2}{p}}\lvert\nabla Q\rvert^{2}\eta^{2}d\mu+2\int_{\Omega}Q^{t+1-\frac{2}{p}}\eta\langle \nabla Q,\nabla\eta\rangle d\mu+a_{1}\int_{\Omega}Q^{t+2}h^{-2}\eta^{2}d\mu\nonumber\\&+2(p-2)\int_{\Omega}Q^{t+1-\frac{4}{p}}\langle\nabla Q,\nabla h\rangle\langle\nabla h,\nabla\eta\rangle\eta d\mu+t(p-2)\int_{\Omega}Q^{t-\frac{4}{p}}\langle\nabla Q,\nabla h \rangle^{2}\eta^{2}d\mu\nonumber\\ \leq&\frac{\sigma pm(p-1)}{m(p-\sigma-1)+2(p-1)}\int_{\Omega}Q^{t+1-\frac{2}{p}}h^{-1}\eta^{2}\langle\nabla Q\nabla h\rangle d\mu\notag\\
&+p(m-1)K\int_{\Omega}Q^{t+2-\frac{2}{p}}\eta^{2} d\mu.
\end{align}

For convenience, we divide the range of $p$ into two cases as follows.

{\em Case 1}: $p\geq2$, so \eqref{Prel-11} can be written as
\begin{align}\label{Prel-12}
&t\int_{\Omega}Q^{t-\frac{2}{p}}\lvert\nabla Q\rvert^{2}\eta^{2}d\mu+a_{1}\int_{\Omega}Q^{t+2}h^{-2}\eta^{2}d\mu\nonumber\\ \leq&\frac{\sigma pm(p-1)}{m(p-\sigma-1)+2(p-1)}\int_{\Omega}Q^{t+1-\frac{2}{p}}h^{-1}\eta^{2}\langle\nabla Q\nabla h\rangle d\mu+p(m-1)K\int_{\Omega}Q^{t+2-\frac{2}{p}}\eta^{2} d\mu\nonumber\\&+2(p-1)\int_{\Omega}Q^{t+1-\frac{2}{p}}\lvert \nabla Q\rvert\lvert\nabla\eta\rvert\eta d\mu.
\end{align}
Using the Cauchy inequality, we have
$$
2(p-1)Q^{t+1-\frac{2}{p}}\lvert\nabla Q\rvert\lvert\nabla\eta\rvert\eta\leq\frac{t}{4}Q^{t-\frac{2}{p}}\lvert\nabla Q\rvert^{2}\eta^{2}+\frac{4(p-1)^{2}}{t}Q^{t+2-\frac{2}{p}}\lvert\nabla\eta\rvert^{2}
$$
and
\begin{align}
&\frac{\sigma pm(p-1)}{m(p-\sigma-1)+2(p-1)}Q^{t+1-\frac{2}{p}}h^{-1}\eta^{2}\langle\nabla Q\nabla h\rangle\nonumber\\
\leq&\frac{t}{4}Q^{t-\frac{2}{p}}\lvert\nabla Q\rvert^{2}\eta^{2}+\frac{\sigma^{2} p^{2}m^{2}(p-1)^{2}}{t[m(p-\sigma-1)+2(p-1)]^{2}}Q^{t+2}h^{-2}\eta^{2}\nonumber.
\end{align}
Inserting the above two inequalities into \eqref{Prel-12} and choosing $t$ large enough such that
$$\frac{a_{1}}{2}-\frac{\sigma^{2} p^{2}m^{2}(p-1)^{2}}{t[m(p-\sigma-1)+2(p-1)]^{2}}>0,$$
then we obtain
\begin{align}\label{Prel-13}
&\frac{t}{2}\int_{\Omega}Q^{t-\frac{2}{p}}\lvert\nabla Q\rvert^{2}\eta^{2}d\mu+\frac{a_{1}}{2}\int_{\Omega}Q^{t+2}h^{-2}\eta^{2}d\mu\nonumber\\ \leq&\frac{4(p-1)^{2}}{t}\int_{\Omega}Q^{t+2-\frac{2}{p}}\lvert\nabla\eta\rvert^{2}d\mu+p(m-1)K\int_{\Omega}Q^{t+2-\frac{2}{p}}\eta^{2} d\mu.
\end{align}

On the other hand, by the Cauchy inequality, we have
\begin{align}\label{Prel-14}
\lvert\nabla( Q^{\frac{t}{2}+1-\frac{1}{p}}\eta)\rvert^{2}\leq2(\frac{t}{2}+1-\frac{1}{p})^{2}Q^{t-\frac{2}{p}}\lvert\nabla Q\rvert^{2}\eta^{2}+2Q^{t+2-\frac{2}{p}}\lvert\nabla\eta\rvert^{2},
\end{align}
so \eqref{Prel-13} can be written as
\begin{align}\label{Prel-15}
&\frac{a_{1}}{2}\int_{\Omega}Q^{t+2}h^{-2}\eta^{2}d\mu+\frac{t}{4(\frac{t}{2}+1-\frac{1}{p})^{2}}\int_{\Omega}\lvert\nabla (Q^{\frac{t}{2}+1-\frac{1}{p}}\eta)\rvert^{2}d\mu\nonumber\\\leq&\frac{4(p-1)^{2}}{t}\int_{\Omega}Q^{t+2-\frac{2}{p}}\lvert\nabla\eta\rvert^{2}d\mu+p(m-1)K\int_{\Omega}Q^{t+2-\frac{2}{p}}\eta^{2}d\mu
\nonumber\\&+\frac{t}{2(\frac{t}{2}+1-\frac{1}{p})^{2}}\int_{\Omega}Q^{t+2-\frac{2}{p}}\lvert\nabla\eta\rvert^{2}d\mu.
\end{align}
We choose $a_{2}$ and $a_{3}$ depending on $m,p$ such that
$$\frac{a_{2}}{t}\leq\frac{t}{4(\frac{t}{2}+1-\frac{1}{p})^{2}}$$
and
$$\frac{4(p-1)^{2}}{t}+\frac{t}{2(\frac{t}{2}+1-\frac{1}{p})^{2}}\leq\frac{a_{3}}{t},$$
then \eqref{Prel-15} gives
\begin{align}\label{Prel-16}
&\frac{a_{1}}{2}\int_{\Omega}Q^{t+2}h^{-2}\eta^{2}d\mu+\frac{a_{2}}{t}\int_{\Omega}\lvert\nabla (Q^{\frac{t}{2}+1-\frac{1}{p}}\eta)\rvert^{2}d\mu\nonumber\\\leq&\frac{a_{3}}{t}\int_{\Omega}Q^{t+2-\frac{2}{p}}\lvert\nabla\eta\rvert^{2}d\mu+p(m-1)K\int_{\Omega}Q^{t+2-\frac{2}{p}}\eta^{2}d\mu.
\end{align}

Replacing $\varphi$ with $Q^{\frac{t}{2}+1-\frac{1}{p}}\eta$ in \eqref{Prel-7} gives
\begin{align}
&e^{-C(m)(1+\sqrt{K}R)}V_{f}^{\frac{2}{m}}R^{-2}\Big(\int_{\Omega}Q^{\frac{2m}{m-2}(\frac{t}{2}+1-
\frac{1}{p})}\eta^\frac{2m}{m-2}\,d\mu\Big)^{\frac{m-2}{m}}\nonumber\\ \leq& \int_{\Omega}|\nabla
(Q^{\frac{t}{2}+1-\frac{1}{p}}\eta)|^2\,d\mu+R^{-2}\int_{\Omega}Q^{t+2-\frac{2}{p}}\eta^2\,d\mu.\nonumber
\end{align}
Substituting the above inequality into \eqref{Prel-16} yields
\begin{align}\label{Prel-17}
&\frac{a_{1}}{2}\int_{\Omega}Q^{t+2}h^{-2}\eta^{2}d\mu+\frac{a_{2}}{t}e^{-C(m)(1+\sqrt{K}R)}V_{f}^{\frac{2}{m}}R^{-2}\Big(\int_{\Omega}Q^{\frac{2m}{m-2}(\frac{t}{2}+1-
\frac{1}{p})}\eta^\frac{2m}{m-2}\,d\mu\Big)^{\frac{m-2}{m}}\nonumber\\\leq&\frac{a_{3}}{t}\int_{\Omega}Q^{t+2-\frac{2}{p}}\lvert\nabla\eta\rvert^{2}d\mu+
p(m-1)K\int_{\Omega}Q^{t+2-\frac{2}{p}}\eta^{2}d\mu
+\frac{a_{2}}{t}R^{-2}\int_{\Omega}Q^{t+2-\frac{2}{p}}\eta^{2}d\mu.
\end{align}
Now we set $t_{0}=C_{1}(m,p)(1+\sqrt{K}R)$, where
$$C_{1}(m,p)=\mathrm{max}\Big\{C(m), \frac{2\sigma^{2} p^{2}m^{2}(p-1)^{2}}{a_{1}[m(p-\sigma-1)+2(p-1)]^{2}}, 2\Big\}.$$
By the definition of $t_{0}$ it is easy to see that for $t>t_{0}$, the inequality
$$\frac{a_{1}}{2}-\frac{\sigma^{2} p^{2}m^{2}(p-1)^{2}}{t[m(p-\sigma-1)+2(p-1)]^{2}}>0$$
holds, and there exists $a_{4}>0$ such that
$$p(m-1)KR^{2}+\frac{a_{2}}{t}\leq a_{4}t_{0}^{2}=a_{4}C_{1}^{2}(m,p)(1+\sqrt{K}R)^{2}.$$
According to \eqref{Prel-17} and above inequalities, we obtain
\begin{align}\label{Prel-18}
&\frac{a_{1}}{2}\int_{\Omega}Q^{t+2}h^{-2}\eta^{2}d\mu+\frac{a_{2}}{t}e^{-C(m)(1+\sqrt{K}R)}V_{f}^{\frac{2}{m}}R^{-2}\Big(\int_{\Omega} Q^{(\frac{t}{2}+1-\frac{1}{p})\frac{2m}{m-2}}\eta^{\frac{2m}{m-2}}d\mu\Big)^{\frac{m-2}{m}}\nonumber\\\leq&\frac{a_{3}}{t}\int_{\Omega}Q^{t+2-\frac{2}{p}}\lvert\nabla\eta\rvert^{2}d\mu
+a_{4}t_{0}^{2}R^{-2}\int_{\Omega}Q^{t+2-\frac{2}{p}}\eta^{2}d\mu.
\end{align}
Since $h=u^{\frac{m[p-(\sigma+1)]+2(p-1)}{(m+2)(p-1)}}\leq M_{0}^{\frac{m[p-(\sigma+1)]+2(p-1)}{(m+2)(p-1)}}$, we have
\begin{align}\label{Prel-19}
&\frac{a_{5}}{2}\int_{\Omega}Q^{t+2}\eta^{2}d\mu+\frac{a_{2}}{t}e^{-C(m)(1+\sqrt{K}R)}V_{f}^{\frac{2}{m}}R^{-2}\Big(\int_{\Omega} Q^{(\frac{t}{2}+1-\frac{1}{p})\frac{2m}{m-2}}\eta^{\frac{2m}{m-2}}d\mu\Big)^{\frac{m-2}{m}}\nonumber\\\leq&\frac{a_{3}}{t}\int_{\Omega}Q^{t+2-\frac{2}{p}}\lvert\nabla\eta\rvert^{2}d\mu
+a_{4}t_{0}^{2}R^{-2}\int_{\Omega}Q^{t+2-\frac{2}{p}}\eta^{2}d\mu,
\end{align}
where $a_{5}=\frac{a_{1}}{M_{0}^{\frac{2m[p-(\sigma+1)]+4(p-1)}{(m+2)(p-1)}}}$.

{\em Case 2}: $1<p<2$, then \eqref{Prel-11} can be written as
\begin{align}\label{add-1}
&t(p-1)\int_{\Omega}Q^{t-\frac{2}{p}}\lvert\nabla Q\rvert^{2}\eta^{2}d\mu+a_{1}\int_{\Omega}Q^{t+2}h^{-2}\eta^{2}d\mu\nonumber\\
\leq&\frac{\sigma pm(p-1)}{m(p-\sigma-1)+2(p-1)}\int_{\Omega}Q^{t+1-\frac{2}{p}}h^{-1}\eta^{2}\langle\nabla Q\nabla h\rangle d\mu\notag\\
&+p(m-1)K\int_{\Omega}Q^{t+2-\frac{2}{p}}\eta^{2} d\mu+2(p-1)\int_{\Omega} Q^{t+1-\frac{2}{p}}|\nabla Q||\nabla\eta|.
\end{align}
Using the Cauchy inequality, we have
$$
2(p-1)Q^{t+1-\frac{2}{p}}\lvert\nabla Q\rvert\lvert\nabla\eta\rvert\eta\leq\frac{t(p-1)}{4}Q^{t-\frac{2}{p}}\lvert\nabla Q\rvert^{2}\eta^{2}+\frac{4(p-1)}{t}Q^{t+2-\frac{2}{p}}\lvert\nabla\eta\rvert^{2}
$$
and
\begin{align}
&\frac{\sigma pm(p-1)}{m(p-\sigma-1)+2(p-1)}Q^{t+1-\frac{2}{p}}h^{-1}\eta^{2}(\nabla Q\nabla h)\nonumber\\
\leq&\frac{t(p-1)}{4}Q^{t-\frac{2}{p}}\lvert\nabla Q\rvert^{2}\eta^{2}+\frac{\sigma^{2} p^{2}m^{2}(p-1)}{t[m(p-\sigma-1)+2(p-1)]^{2}}Q^{t+2}h^{-2}\eta^{2}\nonumber.
\end{align}
Inserting the above two inequalities into \eqref{add-1} and choosing $t$ large enough such that
$$\frac{a_{1}}{2}-\frac{\sigma^{2} p^{2}m^{2}(p-1)}{t[m(p-\sigma-1)+2(p-1)]^{2}}>0,$$
then we obtain
\begin{align}\label{add-2}
&\frac{t(p-1)}{2}\int_{\Omega}Q^{t-\frac{2}{p}}\lvert\nabla Q\rvert^{2}\eta^{2}d\mu+\frac{a_{1}}{2}\int_{\Omega}Q^{t+2}h^{-2}\eta^{2}d\mu\nonumber\\ \leq&\frac{4(p-1)}{t}\int_{\Omega}Q^{t+2-\frac{2}{p}}\lvert\nabla\eta\rvert^{2}d\mu+p(m-1)K\int_{\Omega}Q^{t+2-\frac{2}{p}}\eta^{2} d\mu.
\end{align}

On the other hand, by the Cauchy inequality, we have
\begin{align}\label{add-3}
\lvert\nabla(Q^{\frac{t}{2}+1-\frac{1}{p}}\eta)\rvert^{2}\leq2(\frac{t}{2}+1-\frac{1}{p})^{2}Q^{t-\frac{2}{p}}\lvert\nabla Q\rvert^{2}\eta^{2}+2Q^{t+2-\frac{2}{p}}\lvert\nabla\eta\rvert^{2},
\end{align}
so \eqref{add-2} can be written as
\begin{align}\label{add-4}
&\frac{a_{1}}{2}\int_{\Omega}Q^{t+2}h^{-2}\eta^{2}d\mu+\frac{t(p-1)}{4(\frac{t}{2}+1-\frac{1}{p})^{2}}\int_{\Omega}\lvert\nabla (Q^{\frac{t}{2}+1-\frac{1}{p}}\eta)\rvert^{2}d\mu\nonumber\\\leq&\frac{4(p-1)}{t}\int_{\Omega}Q^{t+2-\frac{2}{p}}\lvert\nabla\eta\rvert^{2}d\mu+p(m-1)K\int_{\Omega}Q^{t+2-\frac{2}{p}}\eta^{2}d\mu
\nonumber\\&+\frac{t(p-1)}{2(\frac{t}{2}+1-\frac{1}{p})^{2}}\int_{\Omega}Q^{t+2-\frac{2}{p}}\lvert\nabla\eta\rvert^{2}d\mu.
\end{align}
We choose $b_{2}$ and $b_{3}$ depending on $m,p$ such that
$$\frac{b_{2}}{t}\leq\frac{t(p-1)}{4(\frac{t}{2}+1-\frac{1}{p})^{2}}$$
and
$$\frac{4(p-1)}{t}+\frac{t(p-1)}{2(\frac{t}{2}+1-\frac{1}{p})^{2}}\leq\frac{b_{3}}{t},$$
then \eqref{add-4} gives
\begin{align}\label{add-5}
&\frac{a_{1}}{2}\int_{\Omega}Q^{t+2}h^{-2}\eta^{2}d\mu+\frac{b_{2}}{t}\int_{\Omega}\lvert\nabla (Q^{\frac{t}{2}+1-\frac{1}{p}}\eta)\rvert^{2}d\mu\nonumber\\\leq&\frac{b_{3}}{t}\int_{\Omega}Q^{t+2-\frac{2}{p}}
\lvert\nabla\eta\rvert^{2}d\mu+p(m-1)K\int_{\Omega}Q^{t+2-\frac{2}{p}}\eta^{2}d\mu.
\end{align}

Next, the argument is identical to that following \eqref{Prel-16} in Case 1, so we have
\begin{align}\label{add-6}
&\frac{a_{5}}{2}\int_{\Omega}Q^{t+2}\eta^{2}d\mu+\frac{b_{2}}{t}e^{-C(m)(1+\sqrt{K}R)}V_{f}^{\frac{2}{m}}R^{-2}\Big(\int_{\Omega} Q^{(\frac{t}{2}+1-\frac{1}{p})\frac{2m}{m-2}}\eta^{\frac{2m}{m-2}}d\mu\Big)^{\frac{m-2}{m}}\nonumber\\\leq&\frac{b_{3}}{t}\int_{\Omega}Q^{t+2-\frac{2}{p}}\lvert\nabla\eta\rvert^{2}d\mu
+b_{4}t_{0}^{2}R^{-2}\int_{\Omega}Q^{t+2-\frac{2}{p}}\eta^{2}d\mu,
\end{align}
where $a_{5}=\frac{a_{1}}{M_{0}^{\frac{2m[p-(\sigma+1)]+4(p-1)}{(m+2)(p-1)}}}$.

Combining the above two cases, letting $c_{2}=\min\{a_{2},b_{2}\}$, $c_{3}=\max\{a_{3},b_{3}\}$ and $c_{4}=\max\{a_{4},b_{4}\}$, the proof of Lemma \ref{lem-2} is finished.

\begin{lem}\label{lem-3}
Let $(M^{n},g,d\mu)$ be an $n$-dimensional $(n>2)$ complete smooth metric measure space with $\mathrm{Ric}_{f,m}\geq-(m-1)Kg$, where $K$ is a nonnegative constant and $\Omega=B_{x_{0}}(R)\subset M$ be a geodesic ball. If the constants $a$, $p$, $\sigma$ conditions in Lemma \ref{lem-1}, then for $\beta=(t_{0}+2-\frac{2}{p})\frac{m}{m-2}$, there exists a nonnegative constant $a_{8}>0$ such that
\begin{align}\label{Prel-20}
\Big(\int_{B_{x_{0}}(\frac{3R}{4})}Q^{\beta}d\mu\Big)^{\frac{1}{\beta}}\leq a_{8}V_{f}^{\frac{1}{\beta}}(\frac{t_{0}}{R})^{p}.
\end{align}
\end{lem}

\proof
Now letting $t=t_{0}$ in \eqref{Prel-8}, we have
\begin{align}\label{Prel-21}
&\frac{a_{5}}{2}\int_{\Omega}Q^{t_{0}+2}\eta^{2}d\mu+\frac{c_{2}}{t_{0}}e^{-t_{0}}V_{f}^{\frac{2}{m}}R^{-2}\Big(\int_{\Omega} Q^{(\frac{t_{0}}{2}+1-\frac{1}{p})\frac{2m}{m-2}}\eta^{\frac{2m}{m-2}}d\mu\Big)^{\frac{m-2}{m}}\nonumber\\\leq&\frac{c_{3}}{t_{0}}\int_{\Omega}Q^{t_{0}+2-\frac{2}{p}}\lvert\nabla\eta\rvert^{2}d\mu
+c_{4}t_{0}^{2}R^{-2}\int_{\Omega}Q^{t_{0}+2-\frac{2}{p}}\eta^{2}d\mu.
\end{align}

Now we let $D=\Big\{x\in\Omega|Q(x)\geq(\frac{4c_{4}}{a_{5}})^{\frac{p}{2}}(\frac{t_{0}}{R})^{p}\Big\}$. Hence we have
\begin{align}\label{Prel-22}
&c_{4}t_{0}^{2}R^{-2}\int_{\Omega}Q^{t_{0}+2-\frac{2}{p}}\eta^{2}d\mu\nonumber\\=&c_{4}t_{0}^{2}R^{-2}\int_{D}Q^{t_{0}+2-\frac{2}{p}}\eta^{2}d\mu+c_{4}t_{0}^{2}R^{-2}\int_{\Omega \backslash D}Q^{t_{0}+2-\frac{2}{p}}\eta^{2}d\mu\nonumber\\ \leq&\frac{a_{5}}{4}\int_{\Omega}Q^{t_{0}+2}\eta^{2}d\mu+c_{4}t_{0}^{2}R^{-2}\big(\frac{4c_{4}}{a_{5}}\big)^{\frac{p}{2}(t_{0}+2)-1}\big(\frac{t_{0}}{R}\big)^{p(t_{0}+2)-2}V_{f}.
\end{align}
Combining \eqref{Prel-21} and \eqref{Prel-22}, we have
\begin{align}\label{Prel-23}
&\frac{a_{5}}{4}\int_{\Omega}Q^{t_{0}+2}\eta^{2}d\mu+\frac{c_{2}}{t_{0}}e^{-t_{0}}V_{f}^{\frac{2}{m}}R^{-2}\Big(\int_{\Omega} Q^{(\frac{t_{0}}{2}+1-\frac{1}{p})\frac{2m}{m-2}}\eta^{\frac{2m}{m-2}}d\mu\Big)^{\frac{m-2}{m}}\nonumber\\\leq&\frac{c_{3}}{t_{0}}\int_{\Omega}Q^{t_{0}+2-\frac{2}{p}}\lvert\nabla
\eta\rvert^{2}d\mu
+c_{4}t_{0}^{2}R^{-2}\big(\frac{4c_{4}}{a_{5}}\big)^{\frac{p}{2}(t_{0}+2)-1}\big(\frac{t_{0}}{R}\big)^{p(t_{0}+2)-2}V_{f}.
\end{align}
Assume $0\leq\xi\leq1, \ \xi\equiv1$ in $B_{x_{0}}(\frac{3R}{4})$ and $\lvert\nabla\xi\rvert\leq\frac{C}{R}$. Let $\eta=\xi^{\frac{p(t_{0}+2)}{2}}$. By a direct calculation, we have
\begin{align}\label{Prel-24}
c_{3}R^{2}\lvert\nabla\eta\rvert^{2}\leq \frac{c_{3}C^{2}p^2}{4}(t_{0}+2)^{2}\eta^{\frac{2p(t_{0}+2)-4}{p(t_{0}+2)}}\leq a_{6}t_{0}^{2}\eta^{\frac{2p(t_{0}+2)-4}{p(t_{0}+2)}}
\end{align}
and
\begin{align}\label{Prel-25}
&\frac{c_{3}}{t_{0}}\int_{\Omega}Q^{t_{0}+2-\frac{2}{p}}\lvert\nabla\eta\rvert^{2}d\mu\nonumber\\ \leq&\frac{a_{6}t_{0}}{R^{2}}\int_{\Omega}\Big(Q^{t_{0}+2-\frac{2}{p}}\eta^{\frac{2p(t_{0}+2)-4}
{p(t_{0}+2)}}(e^{-f})^{\frac{t_{0}+2-\frac{2}{p}}{t_{0}+2}}\Big)\Big((e^{-f})^{\frac{\frac{2}{p}}{t_{0}+2}}\Big)dv\nonumber\\
\leq&\frac{a_{6}t_{0}}{R^{2}}\Big(\int_{\Omega}Q^{t_{0}+2}\eta^{2}d\mu\Big)^{\frac{t_{0}+2-\frac{2}{p}}{t_{0}+2}}\Big(\int_{\Omega}d\mu\Big)^{\frac{\frac{2}{p}}{t_{0}+2}}\nonumber\\
=&\frac{a_{6}t_{0}}{R^{2}}\Big(\int_{\Omega}Q^{t_{0}+2}\eta^{2}d\mu\Big)^{\frac{t_{0}+2-\frac{2}{p}}{t_{0}+2}}V_{f}^{\frac{\frac{2}{p}}{t_{0}+2}},
\end{align}
where in the second inequality we use the H{\"o}lder inequality. Then, using the Young's inequality, we obtain
\begin{align}\label{Prel-26}
&\frac{a_{6}t_{0}}{R^{2}}\Big(\int_{\Omega}Q^{t_{0}+2}\eta^{2}d\mu\Big)^{\frac{t_{0}+2-\frac{2}{p}}{t_{0}+2}}V_{f}^{\frac{\frac{2}{p}}{t_{0}+2}}\nonumber\\
\leq&\frac{a_{6}t_{0}}{R^{2}}\Big[\frac{a_{5}R^{2}}{4a_{6}t_{0}}\frac{t_{0}+2-\frac{2}{p}}{t_{0}+2}
\Big(\big(\int_{\Omega}Q^{t_{0}+2}\eta^{2}d\mu\big)^{\frac{t_{0}+2-\frac{2}{p}}{t_{0}+2}}\Big)^{\frac{t_{0}+2}{{t_{0}+2-\frac{2}{p}}}}\nonumber\\
&+\big(\frac{a_{5}R^{2}}{4a_{6}t_{0}}\big)^{-\frac{t_{0}+2-\frac{2}{p}}{\frac{2}{p}}}\frac{\frac{2}{p}}{t_{0}+2}V_{f}\Big]\nonumber\\
\leq&\frac{a_{5}}{4}\int_{\Omega}Q^{t_{0}+2}\eta^{2}d\mu+\frac{a_{5}}{4}\big(\frac{4a_{6}t_{0}}{a_{5}R^{2}}\big)^{\frac{t_{0}+2}{\frac{2}{p}}}V_{f}.
\end{align}
Therefore, according to \eqref{Prel-23}-\eqref{Prel-26}, we have
\begin{align}\label{Prel-27}
&\Big(\int_{\Omega} Q^{(t_{0}+2-\frac{2}{p})\frac{m}{m-2}}\eta^{\frac{2m}{m-2}}d\mu\Big)^{\frac{m-2}{m}}\nonumber
\\\leq&\frac{t_{0}}{c_{2}}e^{t_{0}}V_{f}^{1-\frac{m}{2}}R^{2}\Big[\frac{c_{4}t_{0}^{2}}{R^{2}}
\Big(\frac{4c_{4}t_{0}^{2}}{a_{5}R^{2}}\Big)^{\frac{p}{2}(t_{0}+2)-1}+\frac{a_{6}t_{0}}{R^{2}}\Big(\frac{4a_{6}t_{0}}{a_{5}R^{2}}\Big)^{\frac{p}{2}(t_{0}+2)-1}\Big]\nonumber\\
\leq&a_{7}^{t_{0}}V_{f}^{1-\frac{2}{m}}t_{0}^{3}\Big(\frac{t_{0}^{2}}{R^{2}}\Big)^{\frac{p}{2}(t_{0}+2)-1},
\end{align}
where
$$a_{7}^{t_{0}}=2\mathrm{max}\Big\{\frac{c_{4}}{c_{2}}e^{t_{0}}\Big(\frac{4c_{4}}{a_{5}}\Big)^{\frac{p}{2}(t_{0}+2)-1}, \ \ \frac{a_{6}}{c_{2}}e^{t_{0}}\Big(\frac{4a_{6}}{a_{5}t_{0}}\Big)^{\frac{p}{2}(t_{0}+2)-1}\Big\},$$
which is equivalent to
\begin{align}\label{Prel-28}
&\Big(\int_{\Omega} Q^{(t_{0}+2-\frac{2}{p})\frac{m}{m-2}}\eta^{\frac{2m}{m-2}}d\mu\Big)^{\frac{m-2}{m}\frac{1}{t_{0}+2-\frac{2}{p}}}\leq a_{8}V_{f}^{\frac{m-2}{m}\frac{1}{t_{0}+2-\frac{2}{p}}}\Big(\frac{t_{0}}{R}\Big)^{p}.
\end{align}
We finish the proof of Lemma \ref{lem-3}.

Therefore, for \eqref{Prel-8}, by ignoring the first term on the left hand, we have
\begin{align}\label{Prel-29}
&\frac{c_{2}}{t}e^{-t_{0}}V_{f}^{\frac{2}{m}}R^{-2}\Big(\int_{\Omega}\lvert Q^{\frac{t}{2}+1-\frac{1}{p}}\eta\rvert^{\frac{2m}{m-2}}d\mu\Big)^{\frac{m-2}{m}}\nonumber\\
\leq& c_{4}t_{0}^{2}R^{-2}\int_{\Omega}Q^{t+2-\frac{2}{p}}\eta^{2}d\mu+\frac{c_{3}}{t}\int_{\Omega}Q^{t+2-\frac{2}{p}}\lvert\nabla\eta\rvert^{2}d\mu\nonumber\\
\leq&\int_{\Omega}\Big(\frac{c_{4}t_{0}^{2}}{R^{2}}\eta^{2}+\frac{c_{3}}{t}\lvert\nabla\eta\rvert^{2}\Big)Q^{t+2-\frac{2}{p}}d\mu,
\end{align}
which is equivalent to
\begin{align}\label{Prel-30}
&\Big(\int_{\Omega}\lvert Q^{\frac{t}{2}+1-\frac{1}{p}}\eta\rvert^{\frac{2m}{m-2}}d\mu\Big)^{\frac{m-2}{m}}\nonumber\\
\leq& V_{f}^{-\frac{2}{m}}e^{t_{0}}\int_{\Omega}\Big(\frac{c_{4}t_{0}^{2}t}{c_{2}}\eta^{2}+\frac{c_{3}R^{2}}{c_{2}}\lvert\nabla\eta\rvert^{2}\Big)Q^{t+2-\frac{2}{p}}d\mu.
\end{align}

Now, we are in the position to apply the Moser iteration. Let
$$\beta_{1}=\beta, \quad \beta_{l+1}=\beta_{l}\frac{m}{m-2}, \quad B_{l}=B(x_{0},\frac{R}{2}+\frac{R}{4^{l}}), \quad l=1,2,\cdots, \quad \Omega_{l}=B_{l}$$
and choose $\eta_{l}\equiv1$ in $B_{l+1},\ \eta_{l}\equiv 0$ in $B_{R}\backslash B_{l}, \lvert\nabla\eta_{l}\rvert\leq\frac{C4^{l}}{R}, 0\leq\eta\leq1$. By letting $t=t_{l}$, $\eta=\eta_{l}$, $t_{l}+2-\frac{2}{p}=\beta_{l}$, we have
\begin{align}\label{Prel-31}
&\Big(\int_{\Omega_{l}}\lvert Q^{\frac{t_{l}}{2}+1-\frac{1}{p}}\eta_{l}\rvert^{\frac{2m}{m-2}}d\mu\Big)^{\frac{m-2}{m}}\nonumber\\
\leq& V_{f}^{-\frac{2}{m}}e^{t_{0}}\int_{\Omega_{l}}\Big(\frac{c_{4}t_{0}^{2}t_{l}}{c_{2}}\eta_{l}^{2}+\frac{c_{3}R^{2}}
{c_{2}}\lvert\nabla\eta_{l}\rvert^{2}\Big)Q^{t_{l}+2-\frac{2}{p}}d\mu\nonumber\\
\leq& V_{f}^{-\frac{2}{m}}e^{t_{0}}\Big(\frac{c_{4}t_{0}^{2}t_{l}}{c_{2}}+\frac{c_{3}}{c_{2}}C^{2}16^{l}\Big)\int_{\Omega_{l}}Q^{t_{l}+2-\frac{2}{p}}d\mu\nonumber\\
\leq& V_{f}^{-\frac{2}{m}}e^{t_{0}}\Big(\frac{c_{4}t_{0}^{2}}{c_{2}}(t_{0}+2-\frac{2}{p})\Big(\frac{m}{m-2}\Big)^{l}+
\frac{c_{3}}{c_{2}}C^{2}16^{l}\Big)\int_{\Omega_{l}}Q^{t_{l}+2-\frac{2}{p}}d\mu\nonumber\\
\leq& V_{f}^{-\frac{2}{m}}e^{t_{0}}\Big(\frac{c_{4}t_{0}^{2}}{c_{2}}(t_{0}+2-\frac{2}{p})16^{l}+\frac{c_{3}}{c_{2}}C^{2}16^{l}\Big)\int_{\Omega_{l}}Q^{t_{l}+2-\frac{2}{p}}d\mu.
\end{align}
It is easy to see that we can find some constant $a_{9}$ such that
\begin{align}\label{Prel-32}
&\Big(\int_{\Omega_{l}}\lvert Q^{\frac{t_{l}}{2}+1-\frac{1}{p}}\eta_{l}\rvert^{\frac{2m}{m-2}}d\mu\Big)^{\frac{m-2}{m}}\leq a_{9}t_{0}^{3}16^{l}V_{f}^{-\frac{2}{m}}e^{t_{0}}\int_{\Omega_{l}}Q^{t_{l}+2-\frac{2}{p}}d\mu,
\end{align}
which is equivalent to
\begin{align}\label{Prel-33}
&\Big(\int_{\Omega_{l}}\lvert Q^{\frac{t_{l}}{2}+1-\frac{1}{p}}\eta_{l}\rvert^{\frac{2m}{m-2}}d\mu\Big)^{\frac{m-2}{m\beta_{l}}}\leq \Big(a_{9}t_{0}^{3}V_{f}^{-\frac{2}{m}}e^{t_{0}}\Big)^{\frac{1}{\beta_{l}}}16^{\frac{1}{\beta_{l}}}
\Big(\int_{\Omega_{l}}Q^{t_{l}+2-\frac{2}{p}}d\mu\Big)^{\frac{1}{\beta_{l}}}.
\end{align}
Thus,
\begin{align}\label{Prel-34}
\Big(\int_{\Omega_{l+1}}Q^{\beta_{l+1}}d\mu\Big)^{\frac{1}{\beta_{l+1}}}\leq\Big(a_{9}t_{0}^{3}V_{f}^{-\frac{2}{m}}e^{t_{0}}\Big)^{\frac{1}{\beta_{l}}}16^{\frac{1}{\beta_{l}}}
\Big(\int_{\Omega_{l}}Q^{t_{l}+2-\frac{2}{p}}d\mu\Big)^{\frac{1}{\beta_{l}}}
\end{align}
which implies
\begin{align}\label{Prel-35}
\Vert Q\Vert_{L_{f}^{\beta_{l+1}}(\Omega_{l+1})}\leq\Big(a_{9}t_{0}^{3}V_{f}^{-\frac{2}{m}}e^{t_{0}}\Big)^{\frac{1}{\beta_{l}}}16^{\frac{1}{\beta_{l}}}\Vert Q\Vert_{L_{f}^{\beta_{l}}(\Omega_{l})}.
\end{align}
By iteration we have
\begin{align}\label{Prel-36}
\Vert Q\Vert_{L_{f}^{\beta_{l+1}}(\Omega_{l+1})}\leq\Big(a_{9}t_{0}^{3}V_{f}^{-\frac{2}{m}}e^{t_{0}}\Big)^{\Sigma_{i=1}^{l}\frac{1}{\beta_{i}}}16^{\Sigma_{i=1}^{l}\frac{i}{\beta_{i}}}\Vert Q\Vert_{L_{f}^{\beta_{1}}(B_{x_{0}}(\frac{3R}{4}))}.
\end{align}
We note that
$$\sum_{i=1}^{\infty}\frac{1}{\beta_{i}}=\frac{m}{2\beta_{1}},\qquad \sum_{i=1}^{\infty}\frac{i}{\beta_{i}}=\frac{m^{2}}{4\beta_{1}},$$
then letting $l\rightarrow\infty$ in \eqref{Prel-36}, we have
\begin{align}
\Vert Q\Vert_{L_{f}^{\infty}(B_{o}(\frac{R}{2}))}\leq a_{10}V_{f}^{-\frac{1}{\beta_{1}}}\Vert Q\Vert_{L_{f}^{\beta_{1}}(B_{x_{0}}(\frac{3R}{4}))},
\end{align}
where $a_{10}\geq(a_{9}t_{0}^{3}e^{t_{0}})^{\frac{m}{2\beta_{1}}}16^{\frac{m^{2}}{4\beta_{1}}}$.
Combining with Lemma \ref{lem-3}, we get
\begin{align}
\lvert\nabla h\rvert\leq\frac{a_{11}(1+\sqrt{K}R)}{R},
\end{align}
where $a_{11}=(a_{8}a_{10})^{\frac{1}{p}}C_{1}$. Because $h=u^{\frac{m[p-(\sigma+1)]+2(p-1)}{(m+2)(p-1)}}$, so we have
\begin{align}
\frac{\lvert\nabla u\rvert}{u^{1-\frac{m[p-(\sigma+1)]+2(p-1)}{(m+2)(p-1)}}}\leq\frac{a_{12}(1+\sqrt{K}R)}{R},
\end{align}
where $a_{12}=\frac{(m+2)(p-1)a_{11}}{m[p-(\sigma+1)]+2(p-1)}$.
Therefore, we complete the proof of Theorem \ref{thm-1}.

\bibliographystyle{Plain}

\end{document}